\documentclass[11pt]{amsart}
\input epsf
\usepackage{latexsym}
\usepackage{amssymb,amsmath,amsthm,amscd, amssymb, 
graphicx, enumerate, verbatim, bm, 
accents}
\usepackage[all]{xy}

\numberwithin{equation}{section}

\newtheorem{lem}[equation]{Lemma}
\newtheorem{prop}[equation]{Proposition}
\newtheorem{thm}[equation]{Theorem}

\newtheorem{quest}[equation]{Question}

\newtheorem{Example}[equation]{Example}

\newtheorem{remark}[equation]{Remark}

\def\co{\colon\thinspace}

\newcommand{\Iso}{\mbox{Iso}}

\newcommand{\e}{\varepsilon}

\newcommand{\II}{\mathbb I}

\def\a{\alpha}

\def\B{\mathcal B}
\def\de{\delta}

\def\o{\omega}

\def\Si{\Sigma}

\def\M{\mathcal M}
\def\K{\mathcal K}

\def\b{\beta}

\def\d{\partial}
\def\k{\kappa}

\def\r{\rho}
\def\o{\omega}

\def\s{\sigma}

\def\Cka{\scalebox{.9}{$C^{k,\a}\!$}}

\def\Ckoa{\scalebox{.9}{$C^{k+1,\a}\!$}}

\def\Cz{\scalebox{.9}{$C^{0}\!$}}

\def\Cinfty{\scalebox{.9}{$C^{\infty}\!$}}

\def\S1{\bf S^1}

\newcommand{\R}{{\mathbb R}}

\makeatletter
\def\equalsfill{$\m@th\mathord=\mkern-7mu
\cleaders\hbox{$\!\mathord=\!$}\hfill
\mkern-7mu\mathord=$}
\makeatother

\begin{document}

\abovedisplayskip=6pt plus3pt minus3pt
\belowdisplayskip=6pt plus3pt minus3pt

\title[Gromov-Hausdorff hyperspace of nonnegatively curved $2$-spheres]
{\bf The Gromov-Hausdorff hyperspace of nonnegatively curved $2$-spheres}

\thanks{\it 2010 Mathematics Subject classification.\rm\ 
Primary 53C21, Secondary 52A20, 53C45, 54B20, 57N20.
\it\ Keywords:\rm\ nonnegative curvature, convex body, hyperspace, 
space of metrics, Gromov-Hausdorff, infinite dimensional topology.}\rm

\author{Igor Belegradek}

\address{Igor Belegradek\\School of Mathematics\\ Georgia Institute of
Technology\\ Atlanta, GA 30332-0160}\email{ib@math.gatech.edu}


\date{}
\begin{abstract} 
We study topological properties of the Gromov-Hausdorff metric
on the set of isometry classes of nonnegatively curved $2$-spheres.
\end{abstract}
\maketitle

\section{Introduction}
\label{sec: intro}

The Gromov-Hausdorff (GH) distance is ubiquitous in
studying families of Riemannian metrics with lower curvature bounds. 
The simplest scenario is when all the metrics in the family
live on the same manifold. 
We call any set of isometry classes of metrics on 
closed $\Cinfty$ manifold $N$ 
equipped with the GH distance
a {\em GH hyperspace of $N$}.  

A metric is {\em intrinsic\,} if the distance between any two points is
the infimum of lengths of curves joining the points.
Any $\Cinfty$ Riemannian metric is intrinsic, and this property is preserved under 
GH limits.
For $\k\in\R$ let 
$\M_{\scriptscriptstyle{\mathrm{curv}\ge \k}}^{\mathrm{\scalebox{.54}{$\mathrm{GH}$}}}(N)$
be the GH hyperspace of intrinsic metrics of curvature $\ge \k$ on $N$. 
Let $\M_{\scriptscriptstyle{\mathrm{sec}\ge \k}}^{\mathrm{\scalebox{.54}{$\mathrm{GH}$}}}(N)$,
$\M_{\scriptscriptstyle{\mathrm{sec}> \k}}^{\mathrm{\scalebox{.54}{$\mathrm{GH}$}}}(N)$
be the GH hyperspaces of $\Cinfty$ Riemannian metrics on $N$
of sectional curvatures $\ge \k$, $>\k$, respectively. 
Topological properties of these GH hyperspaces
are largely a mystery which is why it is more common to give
$\M_{\scriptscriptstyle{\mathrm{sec}> \k}}^{\mathrm{\scalebox{.54}{$\mathrm{GH}$}}}(N)$
the $\Cinfty$ topology resulting in
a stratified space whose strata are Hilbert manifolds~\cite{Bou}.

Our starting point is that for $N=S^2$ and $\k=0$ 
the above GH hyperspaces can be identified 
with the $O(3)$-quotients  of certain
hyperspaces of $\R^3$, see Theorem~\ref{thm: intro GH vs hyperspaces} below. 
This is made possible by the convex surface theory. 

A {\em hyperspace of\,} $\R^3$ is a set of compacta of $\R^3$
equipped with the Hausdorff metric. 
A {\em convex body\,} is a convex set with non-empty interior.
The boundary of any convex body in $\R^3$ inherits an intrinsic metric
of nonnegative curvature, which we call
the {\em boundary metric.} 
A metric that is isometric to the
distance function of a $\Cinfty$ Riemannian metric is {\em intrinsically $\Cinfty$}.
The {\em Steiner point\,} is a way to assign a center to any convex compactum
in $\R^3$ that is continuous, $\Iso(\R^3)$-invariant, and Minkowski linear, and in fact,
these properties characterize the Steiner point~\cite[Theorem 3.3.3]{Sch-book}.
We shall work with the following hyperspaces of $\R^3$:
\[
\K=\{\text{convex compacta in $\R^3$}\}\vspace{4pt}
\]
\[
\K_s=\{\text{convex compacta in $\R^3$ 
with Steiner point at the origin}\}\vspace{4pt}
\]
\[
\K_s^{k\le l}=
\{\text{$D\in\K_s$ with $k\le \dim(D)\le l$}\}\vspace{4pt}
\]
\[
\B_p=
\{\text{convex bodies $D\in \K_s$
with $\Cinfty$ boundary of $\sec>0$}\}\vspace{4pt}
\]
\[
\B_d=
\{\text{convex bodies $D\in \K_s$
with intrinsically $\Cinfty$ boundary metrics}\}.\vspace{4pt}
\]
\[
\B^{k,\a}=
\{\text{$C^{k,\a}$ convex bodies in $\K_s$}\}\text{\quad and \quad $\B^k=\B^{k,0}$.}\vspace{6pt}
\]
One purpose of this paper
is to give an exposition of fundamental (but not widely known) 
results of convex surface theory, which easily imply the following.

\begin{thm}
\label{thm: intro GH vs hyperspaces}
The map $\K_s^{2\le 3}/O(3)\to 
\M_{\scriptscriptstyle{\mathrm{curv}\ge 0}}^{\mathrm{\scalebox{.54}{$\mathrm{GH}$}}}(S^2)$
that assigns to the congruence class of a convex compactum
the isometry class of its boundary surface is a homeomorphism
which restricts to
homeomorphisms $\B_d/O(3)\to
\M_{\scriptscriptstyle{\mathrm{sec}\ge 0}}^{\mathrm{\scalebox{.54}{$\mathrm{GH}$}}}(S^2)$
and $\B_p/O(3)\to 
\M_{\scriptscriptstyle{\mathrm{sec}> 0}}^{\mathrm{\scalebox{.54}{$\mathrm{GH}$}}}(S^2)$.
\end{thm}

Here the {\em boundary surface of a $2$-dimensional convex compactum} $K$ is
the double of $K$ along the boundary with the induced intrinsic metric.

Consider the {\em Hilbert cube\,} $Q=[-1, 1]^\o$ and its {\em radial interior\,} 
$$\Si=\{(t_i)_{i\in \o}\ \text{in}\ Q\,:\, \displaystyle{\sup_{i\in\o}|t_i| < 1} \}.$$
Here $\o$ is the set of nonnegative integers, and
the superscript $\o$ refers to the product of countably many copies of a space.
We have a canonical inclusion $\Si^\o\subset Q^{\,\o}$.
Note that $Q^{\,\o}$ and $Q$ are homeomorphic. 


This paper is a sequel to~\cite{Bel-cb} where the author 
used convex geometry and infinite dimensional topology to
determine
the homeomorphism types of $\K_s$, $\K_s^{2\le 3}$, $\B_p$, and also
derive a number of properties of their \mbox{$O(3)$-quotients}.
In particular,
in~\cite[Section 6]{Bel-cb} 
we isolated some conditions on a hyperspace
$\mathcal D$ with $\B_p\subseteq\mathcal D\subset\B^{1,1}$ that give 
the conclusion of 
Theorem~\ref{thm: intro hyperspace} below
with $\B_d$ replaced by $\mathcal D$. The conditions hold, e.g., 
if $\mathcal D\setminus\B_p$ is $\s$-compact, which includes the case 
$\mathcal D=\B_p$.
Here we verify the conditions for $\mathcal D=\B_d$.

\begin{thm} 
\label{thm: intro hyperspace}
If $E$ is a subset of $Q^{\,\o}\setminus\Sigma^\o$
homeomorphic to suspension of the real projective plane, then
there is a homeomorphism $h\co \K_s^{2\le 3}\to Q^{\,\o}\setminus E$
with $h(\B_d)=\Si^\o$.
\end{thm}

The new ingredient, stated in Theorem~\ref{thm: intro G_de} below,
follows from a version of Cheeger-Gromov compactness theorem.

\begin{thm}
\label{thm: intro G_de}
$\M_{\scriptscriptstyle{\mathrm{sec}\ge 0}}^{\mathrm{\scalebox{.54}{$\mathrm{GH}$}}}(S^2)
\setminus\M_{\scriptscriptstyle{\mathrm{sec}> 0}}^{\mathrm{\scalebox{.54}{$\mathrm{GH}$}}}(S^2)$
is an $F_\s$ subset of 
$\M_{\scriptscriptstyle{\mathrm{sec}\ge 0}}^{\mathrm{\scalebox{.54}{$\mathrm{GH}$}}}(S^2)$
and also
is a countable intersection of $\s$-compact sets.
\end{thm}

Theorem~\ref{thm: intro hyperspace} together with
results in~\cite{Bel-cb} yield a number of topological
properties for the quotients $\K_s/O(3)$, $\B_p/O(3)$, $\B_d/O(3)$,
and hence for the corresponding GH hyperspaces, as summarized below.

\begin{thm}
\label{thm: intro quotients} 
Let $M=\M_{\scriptscriptstyle{\mathrm{curv}\ge 0}}^{\mathrm{\scalebox{.54}{$\mathrm{GH}$}}}(S^2)$
and $M_0$ be the GH hyperspace of the isometry classes in $M$ represented by metrics
with trivial isometry groups.
Let $X$ be $\M_{\scriptscriptstyle{\mathrm{sec}\ge 0}}^{\mathrm{\scalebox{.54}{$\mathrm{GH}$}}}(S^2)$ or
$\M_{\scriptscriptstyle{\mathrm{sec}> 0}}^{\mathrm{\scalebox{.54}{$\mathrm{GH}$}}}(S^2)$,
and let $X_0=X\cap M_0$. Then\vspace{-1pt}
\begin{enumerate}
\item[\textup{(1)}]
$M$ is a locally compact Polish absolute retract. 
\vspace{3pt}
\item[\textup{(2)}] 
$X$ is an absolute retract that is neither Polish nor locally compact. 
\vspace{3pt}
\item[\textup{(3)}]
Any $\s$-compact subset of $X$ has empty interior.
\vspace{3pt}
\item[\textup{(4)}]
$X$ is homotopy dense in $M$, i.e., any continuous map
$Q\to M$ can be uniformly approximated by a continuous  map with image in $X$. 
\vspace{3pt}
\item[\textup{(5)}] $M_0$ is open in $M$. \vspace{3pt}
\vspace{3pt}
\item[\textup{(6)}]
If $L$ is the product of $[0,1)$  and any locally finite simplicial complex 
that is homotopy equivalent to $BO(3)$, 
then there is a homeomorphism $M_0\to L\times Q^{\,\o}$
that takes $X_0$ onto $L\times\Si^\o$.
\vspace{3pt}
\item[\textup{(7)}]
The pairs $(M_0, X_0)$
and $(Q^{\,\o}, \Si^\o)$ are locally homeomorphic, i.e., 
each point of $M_0$ has a neighborhood 
$U\subset M_0$ such that some open embedding
$h\co U\to Q^{\,\o}$ takes 
$U\cap X_0$ onto $h(U)\cap\Si^\o$.
\vspace{3pt}
\item[\textup{(8)}]
$M_0$, $X_0$ are dense but not homotopy dense in $M$, $X$, respectively.
\vspace{3pt}
\item[\textup{(9)}]
$\M_{\scriptscriptstyle{\mathrm{sec}\ge \k}}^{\mathrm{\scalebox{.54}{$\mathrm{GH}$}}}(S^2)$
and $\M_{\scriptscriptstyle{\mathrm{sec}> \k}}^{\mathrm{\scalebox{.54}{$\mathrm{GH}$}}}(S^2)$
are weakly contractible for every $\k>0$.
\end{enumerate}
\end{thm}

Let us supply some context for various items in Theorem~\ref{thm: intro quotients} :

(1)--(2) 
We refer 
to~\cite{Bor-book} for background on absolute retracts (AR) and 
absolute neighborhood retracts (ANR), and only mention here some basic
facts. Any open subset of an ANR is an ANR. Being an AR is equivalent to being
a contractible ANR. Any ANR is locally 
contractible, i.e., any neighborhood $U$ of every point
contains a neighborhood $V$ of the same point such that the inclusion
$V\to U$ is null-homotopic.
Any ANR is homotopy equivalent to a CW complex.

(4) Another definition of a homotopy dense subset $A\subset B$ is that 
there is a homotopy $h\co B\times [0,1]\to B$ with $h(b,0)=b$ and $h(b,t)\in A$ for $t>0$.
The two definitions are equivalent when $B$ is an ANR~\cite[Exercise 10 in Section 1.2]{BRZ-book}.

(5)--(7) The Slice Theorem for compact Lie group 
actions~\cite[Corollary II.5.5]{Bre-book} implies that
$M_0$ is open in $M$ and the restriction of the
orbit map $\K_s^{2\le 3}\to\K_s^{2\le 3}/O(3)$ to the principal orbit
$\accentset{\circ}{\K}^{2\le 3}_s$
is a principal $O(3)$-bundle whose base is
homeomorphic to $M_0$. Similarly, $X_0$ is the base of 
a principal $O(3)$-bundle. 
By~\cite[Lemma 8.2]{Bel-cb} the principal orbit 
$\accentset{\circ}{\B}_p$ for the $O(3)$-action on $\K_s$ is homotopy dense in $\K_s$, and hence
the total spaces of the above principal bundles 
are contractible. Thus $M_0$, $X_0$ are homotopy equivalent to $BO(3)$,
the Grassmanian of $3$-planes in $\R^\o$.
The claims (6)--(7) follow from the main results of~\cite{Bel-cb} and Theorem~\ref{thm: intro hyperspace}.

(8) has a curious interpretation that there is no continuous ``destroy the 
symmetry map'' that would instantly push $M$ into $M_0$, or $X$ into $X_0$.

(9) The contractibility of these GH hyperspaces follow from the contractibility
of $\M_{\scriptscriptstyle{\mathrm{sec}>0}}^{\mathrm{\scalebox{.54}{$\mathrm{GH}$}}}(S^2)$
and a rescaling argument.

In~\cite{Bel-cb} the reader can find a number of 
open questions about the above GH hyperspaces, disguised as $O(3)$-orbit
spaces of hyperspaces of $\R^3$. For example, 
it is unknown whether
$\M_{\scriptscriptstyle{\mathrm{curv}\ge 0}}^{\mathrm{\scalebox{.54}{$\mathrm{GH}$}}}(S^2)$
is a $Q$-manifold, which by Theorem~\ref{thm: intro quotients} is equivalent to
the following.

\begin{quest}
Is $\M_{\scriptscriptstyle{\mathrm{curv}\ge 0}}^{\mathrm{\scalebox{.54}{$\mathrm{GH}$}}}(S^2)$
topologically homogeneous?
\end{quest}

A space is {\em topologically homogeneous\,} if its homeomorphism group acts
transitively.

Theorem~\ref{thm: intro GH vs hyperspaces} is proven in Section~\ref{sec: conv surfaces}
while the other main results are justified in Section~\ref{sec: proofs}. 
In Section~\ref{sec: remarks} we offer some remarks about the hyperspace $\B_d$
whose structure is still quite mysterious.

\section{Spaces on convex surfaces}
\label{sec: conv surfaces} 

In this section we review some fundamental properties of convex surfaces
and prove Theorem~\ref{thm: intro GH vs hyperspaces}.

Two subsets of $\R^3$ are {\em $\de$-congruent\,} if 
some isometry of $\R^3$ takes one subset
within the $\de$-neighborhood of the other one;
if $\de=0$ we call the subsets {\em congruent}.
A homeomorphism $f\co (A, d_A)\to (B, d_B)$ of metric spaces  is a {\em $\de$-isometry}
if $|d_B(f(x),f(y))-d_A(x,y)|<\de$ for any $x,y\in A$. 
If $\de$ is small we use the terms {\em nearly congruent} and
{\em nearly isometric}.

A {\em convex surface\,} is either the boundary of a convex body $B\subset\R^3$ 
or the double $DK$ of a $2$-dimensional convex compactum $K\subset\R^3$
along the identity map of $\d K$, each
with the induced intrinsic metric.
We refer to these two alternatives as the {\em non-degenerate\,}
and the {\em degenerate} convex surfaces, call their intrinsic
metrics the {\em boundary metrics}, and say that they {\em bound\,} $B$, $K$,
respectively. With this definition any convex surface is homeomorphic to $S^2$.

The intrinsic metric on a degenerate surface $DK$ can be canonically approximated
by the boundary metric of the right cylinder with base $K$ and small height.  

Each convex surface bounds a unique convex compactum in $\R^3$
which has dimension $2$ if the surface is degenerate and dimension $3$
otherwise. If two such convex compacta $K_1$, $K_2$ are Hausdorff close, then the corresponding convex surfaces are nearly isometric. (For non-degenerate convex surfaces
this is proved in~\cite[Lemma 10.2.7]{BBI} and the degenerate case
reduces to the non-degenerate one 
by approximating $DK$ with the cylinder as above).

Alexandrov, see~\cite{Alex-1948} 
or~\cite[pp. 112 and 399]{Ale-conv-surfaces} showed 
that an intrinsic metric isometric to a $2$-sphere of 
nonnegative curvature 
if and only if it is isometric to a convex surface.
Pogorelov proved in~\cite{Pog-uniq} 
that any two isometric convex surfaces are congruent, 
even though his argument is commonly described as very complicated,
and I hesitate to rely on it.
An easier proof of this result was found by Volkov~\cite{Vol}, 
see~\cite[Section 12.1]{Ale-conv-surfaces} for a reprint
and~\cite[Section 5.2]{BurShe} for an
exposition of Volkov's stability theorem which we discuss below.

Each non-degenerate convex surface has another metric 
obtained by restricting the distance function on $\R^3$;
we call the metric {\em extrinsic}.
If $\Si_1$, $\Si_2$ are non-degenerate convex surfaces with intrinsic metrics
$\r_1$, $\r_2$, and extrinsic metrics $d_1$, $d_2$, and if
$f\co(\Si_1,\r_1)\to(\Si_2,\r_2)$ is an $\e$-isometry, then 
Volkov stability theorem states that 
$f\co(\Si_1, d_1)\to(\Si_2, d_2)$ is an $C_1\e^{\b}$-isometry
where $C_1$ depends onto on diameters of $\r_1$, $\r_2$ and $\b$ is 
a positive universal constant. This easily implies that $\Si_1$, $\Si_2$ 
are nearly congruent, e.g., according to~\cite[Theorem 2.2]{ATV} 
any $\de$-isometry between compacta in $\R^n$
can be approximated by the restriction of an isometry of $\R^n$
with the additive error at most $C_2\sqrt{\de}$ where $C_2$ depends only on $n$
and the diameters of the compacta. 

To extend the result to the case of a degenerate surface $DK$
we replace it with a nearby right cylinder with base $K$, and then
apply Volkov's theorem. 

If the isometry classes of two convex surfaces are GH close, then
the surfaces are nearly isometric,
e.g., by the Perelman stability theorem~\cite{Kap-Perelman-Stab}.
(A less heavy-handed argument is as follows. 
For a convex surface $\Si$ we denote its isometry class by $[\Si ]$.
If $\Si_i$, $\Si$ are convex surfaces such that the sequence
$[\Si_i]\to[\Si]$ in the GH metric, then 
up to congruence $\bigcup_i\Si_i$ has compact closure in $\R^3$
and any limit point of the sequence $\Si_i$ with respect to the Hausdorff metric
is congruent to $\Si$, which by above gives the desired near isometry 
of $\Si$ and $\Si_i$ for large $i$).

The map $r\co\K\to\K_s$
given by $r(D)=D-s(D)$ where $s$ is the Steiner point
descends to a homeomorphism of orbit spaces $\K/\Iso(\R^3)\to\K_s/O(3)$,
see~\cite[Section 4]{Bel-cb}.
Note that the homeomorphism is dimension preserving.

A {\em $C^{k,\a}$ convex body\,} is a convex body whose boundary is 
a $C^{k,\a}$ submanifold of $\R^n$. A function is $C^{k,\a}$ if its
$k$th partial derivatives are 
$\a$-H\"older for $\a\in (0,1]$ and continuous for $\a=0$.
As usual $C^k$ means $C^{k,0}$.

\begin{lem}
\label{lem: C11 bound}
Any convex body $D\in\B_d$ has $C^{1,1}$ boundary that
is $C^\infty$ at points of intrinsically positive curvature.
In particular, if the boundary metric is intrinsically $C^\infty$
of positive sectional curvature, then $D\in\B_p$.
\end{lem}
\begin{proof}
The last statement was proved much earlier by Pogorelov and 
Nirenberg (independently). The boundary $\d D$ is the image
of an isometric embedding of the distance function of 
a $\Cinfty$ nonnegatively curved metric $g$ on $S^2$. Improving on
Nirenberg's method Guan-Li~\cite{GuaLi} and Hong-Zuily~\cite{HonZui}
independently proved that any $C^\infty$ nonnegatively curved 
metric of $S^2$ admits a $C^{1,1}$ isometric embedding into $\R^3$ 
that is $C^\infty$ at points of positive curvature, and moreover
the embedding is the limit of a sequence of $C^\infty$ isometric embeddings of
positively curved metrics on $S^2$. By Hadamard theorem, see e.g.,~\cite[Chapter 2]{Spi-III}, 
the image of an isometric embedding
of positively curved sphere bounds a convex body, and hence the same
is true for the limiting $C^{1,1}$ isometric embedding that induces $g$. 
The limiting convex body is congruent to $D$ by the the above mentioned
results of Pogorelov and Volkov. 
\end{proof}

The above discussion proves Theorem~\ref{thm: intro GH vs hyperspaces}.

\section{Proofs of main results}
\label{sec: proofs}

\begin{proof}[Proof of Theorem~\ref{thm: intro G_de}]
For integers $k\ge 2$, $l\ge 1$ let 
$Q_l^k\subset \M_{\scriptscriptstyle{\mathrm{sec}\ge 0}}^{\mathrm{\scalebox{.54}{$\mathrm{GH}$}}}(S^2)$
be the subset
consisting of isometry classes of metrics whose sectional curvature vanishes 
somewhere, the diameter is in $[0,l]$, 
the injectivity radius is at least $1/l$, and the $\Cz$ norms of
the curvature tensor and of every covariant derivative 
of the curvature tensor of orders $1,\dots, k$ 
is at most $l$.
Its closure $\bar Q_l^k$ in 
$\M_{\scriptscriptstyle{\mathrm{curv}\ge 0}}^{\mathrm{\scalebox{.54}{$\mathrm{GH}$}}}(S^2)$ is 
compact and disjoint from 
$\M_{\scriptscriptstyle{\mathrm{sec}> 0}}^{\mathrm{\scalebox{.54}{$\mathrm{GH}$}}}(S^2)$
because for each $\a\in (0,1)$ any sequence in $Q_l^k$ subconverges in the $\Cka$ topology 
to an isometry class of a $\Ckoa$
Riemannian manifold, see e.g.~\cite[Theorem 2.2]{And-conv}, 
and since $k\ge 2$ the sectional curvature must vanish in the limit.
For each $k$ we clearly have  \[
\M_{\scriptscriptstyle{\mathrm{sec}\ge 0}}^{\mathrm{\scalebox{.54}{$\mathrm{GH}$}}}(S^2)
\setminus 
\M_{\scriptscriptstyle{\mathrm{sec}> 0}}^{\mathrm{\scalebox{.54}{$\mathrm{GH}$}}}(S^2)=
\bigcup_{l\ge 1}\,\bar Q_l^k\cap 
\M_{\scriptscriptstyle{\mathrm{sec}\ge 0}}^{\mathrm{\scalebox{.54}{$\mathrm{GH}$}}}(S^2)
\] which is $F_\s$ in 
$\M_{\scriptscriptstyle{\mathrm{sec}\ge 0}}^{\mathrm{\scalebox{.54}{$\mathrm{GH}$}}}(S^2)$.
The $\s$-compact set $\bigcup_{_{l\in\o}}\! \bar Q_l^k$ in
$\M_{\scriptscriptstyle{\mathrm{curv}\ge 0}}^{\mathrm{\scalebox{.54}{$\mathrm{GH}$}}}(S^2)$
\begin{itemize}\vspace{-2pt} 
\item consists of the isometry classes of $\Ckoa$ Riemannian manifolds,\vspace{2pt} 
\item contains
$\M_{\scriptscriptstyle{\mathrm{sec}\ge 0}}^{\mathrm{\scalebox{.54}{$\mathrm{GH}$}}}(S^2)
\setminus 
\M_{\scriptscriptstyle{\mathrm{sec}> 0}}^{\mathrm{\scalebox{.54}{$\mathrm{GH}$}}}(S^2)$,
\vspace{2pt} 
\item and is disjoint from 
$\M_{\scriptscriptstyle{\mathrm{sec}> 0}}^{\mathrm{\scalebox{.54}{$\mathrm{GH}$}}}(S^2)$.
\vspace{-2pt} 
\end{itemize}
Thus 
$\M_{\scriptscriptstyle{\mathrm{sec}\ge 0}}^{\mathrm{\scalebox{.54}{$\mathrm{GH}$}}}(S^2)
\setminus 
\M_{\scriptscriptstyle{\mathrm{sec}> 0}}^{\mathrm{\scalebox{.54}{$\mathrm{GH}$}}}(S^2)$
equals
$\bigcap_{_{k\ge 2}}\bigcup_{_{l\ge 1}}\! \bar Q_l^k$ as claimed.
\end{proof}

Now the results of~\cite{Bel-cb} can be put together to yield
what we claimed in the introduction. To justify this we are going to use some 
infinite dimensional topology terminology that can be found in~\cite[Section 3]{Bel-cb}.

\begin{proof}[Proof of Theorem~\ref{thm: intro hyperspace}]
First we show that $\B_d$ is homeomorphic to $\Si^\o$. 
By Lemma~\ref{lem: C11 bound} we have $\B_p\subset\B_d\subset\B^{1,1}$, 
hence~\cite[Lemmas 6.1--6.3]{Bel-cb} show that $\B_d$ is an AR with SDAP and also $\s Z$.

The $O(3)$-orbit maps from $\B_d$ and $\B_p$ onto
the sets of congruence classes are continuous and proper.
Taking preimage of a proper continuous map preserves being $F_\s$ and being $\s$-compact
so preimages $\B_p\setminus \B_p$ is $F_\s$ in $\B_d$ and also
is a countable intersection of $\s$-compact sets.
Hence~\cite[Lemmas 6.6 and 6.9]{Bel-cb} imply that $\B_d\in\M_2$
and $\B_d$ is strongly $\M_2$-universal. These properties imply that
$\B_d$ is homeomorphic to $\Si^\o$.

Then the pair $(\K_s^{2\le 3}, \B_d)$ is $(\M_0,\M_s)$-absorbing
by~\cite[Lemma 7.1]{Bel-cb}.

Also~\cite[Lemma 5.2]{Bel-cb} shows that 
$\K_s^{2\le 3}$ is homeomorphic to the complement in $Q^{\,\o}$ of a $Z$-set
homeomorphic to the suspension $SRP^2$ over $RP^2$.
Since $\Si^\o$ is convex and dense in $Q^{\,\o}$, it is
also homotopy dense in $Q^{\,\o}$, see~\cite[Exercise 13 in 1.2]{BRZ-book}.
Hence every compact subset of $Q^{\,\o}\setminus\Si^\o$ is a $Z$-set.
If $E$ is as in the statement of Theorem~\ref{thm: intro hyperspace}, then
by the knotting of $Z$-sets in $Q$-manifolds~\cite[Theorem 1.1.25]{BRZ-book} 
the set $Q^{\,\o}\setminus E$ can be taken to $\K_s^{2\le 3}$ 
by some homeomorphism of $Q^{\,\o}$. The pair $(Q^{\,\o}\setminus E,\Si^\o)$
is $(\M_0,\M_s)$-absorbing by~\cite[Lemma 7.2]{Bel-cb}.
Now the uniqueness of absorbing pairs~\cite[Lemma 7.2]{Bel-cb}
proves Theorem~\ref{thm: intro hyperspace} for $\B_d$. 
The same argument works for $\B_p$.
\end{proof}

\begin{proof}[Proof of Theorem~\ref{thm: intro quotients}]
The statements (1)-(8) of Theorem~\ref{thm: intro quotients} were proved
in~\cite[Section 8--9]{Bel-cb} 
for the $O(3)$-quotients of an arbitrary $O(3)$-invariant hyperspace $X$
that is locally homeomorphic to $\Si^\o$ and such that $\B_p\subset X\subset\K_s$. 
The statement (9) was explained in~\cite[Question (g) of Section 1]{Bel-cb}.
\end{proof}

\section{Remarks on the structure of $\B_d$}
\label{sec: remarks}

The hyperspace $\B_d$, which is the main object of his paper,
is not well-understood, e.g., I suspect that $\B_d$
is not convex but cannot yet prove it. This section
is to shed some light on the properties of $\B_d$.

The moral of Theorem~\ref{thm: intro G_de} is that the awkward features of $\B_d$
disappear in $\B_d/O(3)$, as they should because the condition of being
intrinsically $\Cinfty$ makes much more sense in
$\M_{\scriptscriptstyle{\mathrm{sec}\ge 0}}^{\mathrm{\scalebox{.54}{$\mathrm{GH}$}}}(S^2)$. 

Recall that $\B_p\subset\B^\infty\subset\B_d\subset\B^{1,1}$.
It turns that $\B_d\setminus\B^\infty$ is quite large.

\begin{lem}
\label{lem: C11 approxim}
$\B_d\setminus\B^2$ is dense in $\K_s$.
\end{lem}
\begin{proof}
The convex surface $x_3=f(x_1,x_2)=r^3$, where $r=\sqrt{x_1^2+x_2^2}$, 
is $C^{1,1}$ but not $C^2$.
Its boundary metric is intrinsically $\Cinfty$ because
the components of the metric tensor induced on the graph of $f\co\R^n\to\R$ 
are $g_{ij}=\de_{ij}+\frac{\d f}{\d x_i}\frac{\d f}{\d x_j}$ and
for $f$ as above we have $\frac{\d f}{\d x_i}=3rx_i$.

A small neighborhood of the origin in this surface can be patched as in~\cite{Gho-jdg2001}
at any point of positive curvature of every $\Cinfty$ convex surface
to produce a convex surface that has positive curvature everywhere except at
one point near which it the graph of $f$.
The conditions of Ghomi's patching theorem are satisfied
because in the $x_1,x_2$ local coordinates any positively curved surface
lies above the graph of $g(x_1,x_2)=kr^2$ for some $k>0$,
and hence $kr^2>r^3$ for small $r$. 
Thus $\B_p$ lies in the closure of $\B_d\setminus\B^2$ in $\K_s$, 
and the claim follows by noting that
by Schneider's regularization  $\B_d$ is dense in $\K_s$, see e.g.~\cite[Section 4]{Bel-cb}.
\end{proof}
 
In Theorem~\ref{thm: intro hyperspace} we show that $\B_d$ is homeomorphic to
$\Si^\o$. The same is true for any hyperspace $\B_p\subseteq\mathcal D\subset\B^{1,1}$
such that $\mathcal D\setminus\B_p$ is $\s$-compact~\cite[Theorem 6.10]{Bel-cb}.
Perhaps this conclusion holds for any naturally occurring hyperspace 
$\mathcal D$ with $\B_p\subseteq\mathcal D\subset\B^{1}$, and
while thinking on this problem one wants an example of a hyperspace
that is not homeomorphic to $\Si^\o$. 

In~\cite[Theorem 6.11]{Bel-cb} one finds
a hyperspace $\mathcal D$ with $\B_p\subset\mathcal D\subset\B^{1,1}$ 
such that $\mathcal D\setminus\B_p$ embeds into the Cantor set, 
$\B_p$ is open in $\mathcal D$, and $\mathcal D$ is not a topologically
homogeneous, and in particular, not homeomorphic to $\Si^\o$.
We improve this example as follows.

\begin{prop}
There is a hyperspace $\mathcal D$ with $\B_p\subset\mathcal D\subset\B_d\cap\B^{2,1}$ 
such that $\mathcal D\setminus\B_p$ embeds into the Cantor set, 
$\B_p$ is open in $\mathcal D$, and $\mathcal D$ is not a topologically
homogeneous.
\end{prop}
\begin{proof}
A slight modification of an example in~\cite{Iai-example} 
gives a $3$-dimensional convex body whose boundary is $C^\infty$ except at one point 
$p$ where it is $C^{2,1}$ but not $C^3$, and such that the boundary metric is 
intrinsically $C^\infty$. The curvature vanishes at $p$ and is positive elsewhere. 
Any slight smooth perturbation at a boundary point of positive curvature
gives a body with the same properties, and in particular,
there is a path of such metrics, so by the proof of~\cite[Theorem 6.11]{Bel-cb} we can pick 
$\mathcal D\setminus\B_p$ to be a subset of a Cantor set inside this path.
\end{proof}

\small
\bibliographystyle{amsalpha}
\bibliography{ghs2}

\end{document}